\documentclass[12pt]{article}
\usepackage{layout}
\usepackage[utf8]{inputenc}
\usepackage{graphicx}
\usepackage{amsmath,fullpage}
\usepackage{amsfonts}
\usepackage{amssymb}
\usepackage{amsthm}
\usepackage{xcolor}
\newtheorem{thm}{Theorem}[section]
\newtheorem{lem}{Lemma}[section]

\newtheorem{exam}{Example}[section]
\newtheorem{defn}{Definition}[section]
\newtheorem{cor}{Corollary}[section]
\newtheorem{rmk}{Remark}[section]

\DeclareMathOperator{\Syl}{Syl}
\DeclareMathOperator{\Res}{Res}
\DeclareMathOperator{\sym}{Sym}
\DeclareMathOperator{\diag}{diag}

\title{On dense subsets of matrices with distinct eigenvalues and distinct singular values \thanks{Part of the contents of this article is included in the first author's M.Sc., thesis. }}
\author{Himadri Lal Das \thanks{Department of Mathematics, Indian Institute of Technology Kharagpur, Kharagpur 721302, India. Email: himadrilaldas2014@gmail.com}\  \and M. Rajesh Kannan\thanks{Department of Mathematics, Indian Institute of Technology Kharagpur, Kharagpur 721302, India. Email: rajeshkannan@maths.iitkgp.ac.in, rajeshkannan1.m@gmail.com }
}
\date{\today}
\begin{document}
\maketitle
\baselineskip=0.25in

\begin{abstract}  It is well known that the set of all  $ n \times n $ matrices with distinct eigenvalues is a dense subset of the set of all real or complex  $ n \times n $ matrices. In [Hartfiel, D. J. Dense sets of diagonalizable matrices. Proc. Amer. Math. Soc., 123(6):1669–1672, 1995.], 
 the author established a necessary and sufficient condition for a subspace of the set of all $n \times n$ matrices to have a  dense subset of matrices with distinct eigenvalues. We are interested in finding a few necessary and sufficient conditions for a subset of the set of all $n \times n$ real or complex matrices to have a dense subset of matrices with distinct eigenvalues. Some of our results are generalizing the results of Hartfiel.  Also, we study the existence of dense subsets of matrices with distinct singular values, distinct analytic eigenvalues, and distinct analytic singular values, respectively, in the subsets of the set of all real or complex matrices.
    \end{abstract}



\section{Introduction}
It is well known that the set of all  $ n \times n $ matrices with distinct eigenvalues is a dense subset of the set of all real or complex  $ n \times n $ matrices. An arbitrary subspace of the set of all $n \times n$ matrices may not have a dense subset of matrices with distinct eigenvalues.  For a subset $\Omega$ of the set of all $n \times n$ matrices, let $\Omega_d$ denote the set of all matrices in $\Omega$ with distinct eigenvalue. In \cite{hart}, the author established the following:

\begin{thm}[{\cite[Theorem 1, Corollary 1]{hart}}]\label{hart-main-res}
	If $\Omega$ is a subspace and  $\Omega_d$ is nonempty, then  $\Omega_d$ is dense in $\Omega$. If $\Omega$ is a convex set and  $\Omega_d$ is nonempty, then  $\Omega_d$ is dense in $\Omega$
\end{thm}
The motivation to consider problems of these nature arises in analyzing the behavior of a system \cite{hart-track}.  In \cite{hart-track}, the author considered a  particular case of this result viz., for the set of all stochastic matrices. It is of interest to know whether the counterpart of these results holds for singular values, analytic eigenvalues, and analytic singular values. In \cite{mo-quest}, we found the following question: for a complex square matrix $A$ does there exist a complex symmetric matrix $S$ such that $\Vert S \Vert \leq \epsilon $ such that $A+S$ has only distinct singular values.  We will be able to answer this question with the aid of our results (Remark \ref{mo-ans}).

The first objective of this article is to extend Theorem \ref{hart-main-res} for a larger class of matrices, and weakening the assumptions. For a convex set $\Omega$, we show that, if the closure of $\Omega$ contains a matrix with distinct eigenvalues, then $\Omega_d$ is dense in $\Omega$ (Theorem \ref{a8}). Also, we extend this result for non-convex sets (Theorem \ref{conv-gen-eig}).    The second objective of this article is to study the counterpart of these results for singular values of rectangular matrices.

 For an $n \times n$ matrix $A$,  define a $(0,1)$-matrix $A^{'}$ as follows: $(i,j)$th entry of $A^{'}$ is $1$, if the $(i,j)$th entry of the matrix $A$ is nonzero, and $0$ otherwise. For a $(0,1)$-matrix $P$ of size $n \times n$, define $W_P= \{A:  A^{'}\leq P ~ \mbox{componentwise} \}$.    In \cite{hart}, the author applied  Theorem \ref{hart-main-res} to the subspace $W_P$, and obtained a necessary and sufficient condition, in terms of $P$,  for the subspace $W_P$ to have a dense subset of matrices with distinct eigenvalues. In this article, we  study the counter part of this result for singular values of rectangular matrices (Theorem \ref{zero-patt-sing}).
Also we study the  existence of dense subsets of matrices with distinct analytic eigenvalues and distinct analytic singular values, using the notion of  analytic spectral decomposition and analytic singular value decomposition. For details about analytic spectral decomposition and analytic singular value decomposition, we refer to \cite{sing, ging,gin, kato}. Besides this, we define a class of polynomials  associated with a matrix defined in terms of the entries of the matrix, and we study the same problem for these newly defined class of polynomials viz., existence of dense subset of matrices  with distinct zeros with respect to  polynomials in this class.

The organization of this article is as follows: In Section 2, we collect the needed known definition and results. Section 3  divided into two subsections, in Subsection 3.1, we prove results related to the denseness of the set of matrices with distinct eigenvalues. We also provide a weaker form of Theorem \ref{hart-main-res}. In Subsection 3.2, we extend some of the results of Subsection 3.1 for the set of matrices with distinct singular values, which includes an extension of Theorem 1 of \cite{hart} for singular values of rectangular complex (or real) matrices. In Section 4, we prove the counterpart of some of the results of Section 3 for the analytic eigenvalues and analytic singular values. In Section 5, we introduce a class of functions defined in terms of entries of a given matrix, which includes characteristic polynomial, and prove results  similar to that  of Section 3.

\section{Notation, definition and preliminary results}
Let $\mathbb{R}$ and $\mathbb{C}$ denote the set of all real and complex numbers, respectively. 
Let $M_{m \times n}(\mathbb{F})$ denote the set of all $m \times n$ matrices whose entries are from $\mathbb{F}$, where $\mathbb{F}$  is $\mathbb{R}$ or $\mathbb{C}$. Throughout this paper, we assume $m \geq n$. Let $A \in M_{m\times n}( \mathbb{F})$, with  $m\geq n$. We   call $a_{ii}$, $i=1,2,\dots,n$, the diagonal entries of $A$. For a matrix $A\in M_{m \times n}(\mathbb{R})$,  $A^T$ denotes the transpose of $A$,  and for  $A\in M_{m \times n} (\mathbb{C})$,  $A^{\ast}$ denotes the conjugate transpose of $A$ . For a given matrix $A \in  M_{n \times n}(\mathbb{C})$ the determinant of $A$ is denoted by $\det A$.  If $A$ is a positive semidefinite $n \times n$ matrix, then there exists a unique positive semidefinite $n\times n$ matrix $B$ such that $A=B^2$. Such a matrix $B$ is said to be the square root of $A$, and is denoted by $\sqrt{A}$.

For a matrix $A :=(a_{ij})$ in $M_{m\times n}(\mathbb{C})$, the Frobenius norm of $A$, denoted by $\Vert A \Vert_F$, is defined by $\sqrt{  \sum\limits_{i=1}^m \sum\limits_{j=1}^n \vert a_{ij}\vert^2 }$. Throughout this article, we use the topology induced by the Frobenius norm on the set of all real or complex $m \times n$ matrices. For a matrix $A \in M_{m\times n}(\mathbb{C})$, define $\mathrm{B}(A;\epsilon) = \{B \in M_{m \times n}(\mathbb{C}) : \Vert B - A \Vert < \epsilon\}.$

For $X \subseteq  M_{m \times n}(\mathbb{C})$ (or $M_{m \times n}(\mathbb{R})$), the closure of $X$ is denoted by $\overline{X}$. In a metric space $E$, a point $p\in E$ is called an isolated point in $E$ if $p$ is not a limit point. We call a set $E$ isolated if every point of $E$ is isolated point.
\begin{thm}[Theorem 11.4 (a), \cite{har}]\label{closed} In  $M_{n \times n}(\mathbb{C})$,  the set of matrices that have multiple eigenvalues (at least one eigenvalue of multiplicity 2 or more) is closed.
\end{thm}

%

Let $S$ be a commutative ring with unity. Let $f(x)=\sum\limits_{k=0}^n a_k x^k$ and $g(x)=\sum\limits_{k=0}^m b_k x^k$ be two polynomials in $S[x]$ of degree $n$ and $m$, respectively.

\begin{defn}[Resultant]
The Sylvester matrix of $f(x),g(x) \in S[x]$ is an $(n+m) \times (n+m)$ matrix, denoted by $\Syl(f,g)$,  defined as
\[\Syl(f,g)=
\begin{bmatrix}
a_n&a_{n-1}&.&.&.&a_0&0&.&\dots&.\\
0&a_n&a_{n-1}&.&.&.&a_0&0&\dots&.\\
\hdotsfor{10}\\
.&.&0&a_n&a_{n-1}&.&.&.&\dots&a_0\\
b_m&b_{m-1}&.&.&b_0&0&.&.&\dots&.\\
0&b_m&b_{m-1}&.&.&b_0&0&.&\dots&.\\
\hdotsfor{10}\\
.&.&0&b_m&b_{m-1}&.&.&.&\dots&b_0\\
\end{bmatrix}.
\]
The resultant of two polynomials $f(x),g(x)\in S[x]$ is the determinant of the $\Syl(f,g)$, and is denoted by $\Res(f,g)$.
\end{defn}

\begin{thm}[Theorem 5.7, \cite{jac}] \label{res} Let $F$ be 
a field. If  $f(x)=\sum\limits_{k=0}^n a_k x^k$ and $g(x)=\sum\limits_{k=0}^m b_kx^k$ are two elements in $F[x]$, where $m$ and $ n$ are positive integers, then,  $\Res(f,g)=0$ if and only if either $a_n=0=b_m$ or $f(x)$ and $g(x)$ has a common factor of positive degree in $F[x]$.
\end{thm}

 Hence, if $f(x)$ and $g(x)$ are two non-constant polynomial in $F[x]$, where $F$ is algebraically closed field, then $\Res(f,g)=0$ if and only if $f(x)$ and $g(x)$ has a common root.

A function of the form $$f(x)=\sum_{n=0}^{\infty}c_nx^n$$ or, more generally, $$f(x)=\sum_{n=0}^{\infty}c_n(x-a)^n$$ is called an analytic function, where the domain of the function is an open subset $\mathcal{U}$ of $\mathbb{R}$ or $\mathbb{C}$, and $c_n$ are in $\mathbb{R}$ or $ \mathbb{C}$  \cite{rudin}.

Let $\mathcal{E}$ be a subset of $\mathbb{R}$, and let $f:\mathcal{E} \rightarrow \mathbb{R}$ be a function. If $a$ is an interior point of $\mathcal{E}$,  $f$ is said to be  real analytic at $a$, if there exists an open interval $(a-r,a+r)$ in $\mathcal{E}$ for some $r>0$, such that there exists a power series $\sum_{n=0}^{\infty} c_n (x-a)^n$ centered at $a$, which has radius of convergence greater or equal to $r$, and  converges to $f$ in $(a-r,a+r)$ \cite{tao}.
 \begin{thm}[Theorem 8.5, \cite{rudin}]\label{pow-uniq-rudin}
Suppose the series $\sum\limits_{n=0}^{\infty} a_nx^n$ and $\sum\limits_{n=0}^{\infty} b_n x^n$ converge in the segment $\mathcal{S}=(-R,R)$. Let $\mathcal{E}$ be the set of all $x\in \mathcal{S}$ at which $$\sum_{n=0}^{\infty} a_n x^n= \sum_{n=0}^{\infty} b_n x^n.$$ If $\mathcal{E}$ has a limit point in $\mathcal{S}$, then $a_n=b_n$ for $n=0,1,2, \dots$, and hence $\sum\limits_{n=0}^{\infty} a_n x^n= \sum\limits_{n=0}^{\infty} b_n x^n$ holds for all $x\in \mathcal{S}$.
\end{thm}

Let $M^{\mathbb{R}}_{n\times n}(\mathbb{C})$ denote the set of all matrices having real eigenvalues.
Let $A(t)$ be a family of matrices such that the entries are smoothly  depend on a parameter $t$, for $t \in [a,b]$.  The following theorem is a counter part of Schur's lemma for the matrices $A(t)$ whose entries are analytic functions.
\begin{thm}[ Theorem 1.1 ,\cite{ging}]\label{a9}
Let $A(t)$ be an $n\times n$ matrix function with analytic entries on $[a,b]$, where $-\infty \leq a < b \leq \infty$. If $A(t) \in M^{\mathbb{R}}_{n\times n}(\mathbb{C}) $ for each  $t \in [a,b]$, then there exist an unitary matrix function $U(t)$, which is analytic on $[a,b]$,  such that $$Q(t)=U^{-1}(t)A(t)U(t),$$ where $Q(t)$ is an upper-triangular matrix whose entries are analytic functions of $t$ on $[a,b]$.
\end{thm}

Diagonal entries of $Q(t)$ are called the analytic eigenvalues of $A(t)$.
A singular value decomposition of a matrix $A$ in $M_{m\times n}(\mathbb{C})$ is a factorization $A=U\Sigma V^\ast$, where $U$ is an $m\times m$ unitary matrix, V is an $n\times n$ unitary matrix and $\Sigma=\diag(s_1,s_2,\dots,s_n)$ is an $m\times n$ diagonal matrix, where $m \geq n$. The numbers $s_i$ are  called the singular values. They may be defined to be non negative and to be arranged in non increasing order.
\begin{defn}[Analytic singular value decomposition, \cite{sing}] For a real analytic matrix valued function $E(t):[a,b]\rightarrow M_{m \times n}(\mathbb{R})$, an analytic singular value decomposition is a path of factorization $$E(t)=X(t)S(t)Y(t)^T$$ where $X(t):[a,b]\rightarrow M_{m \times m} (\mathbb{R})$ is orthogonal, $S(t):[a,b]\rightarrow M_{m \times n}(\mathbb{R})$ is diagonal, $Y(t): [a,b]\rightarrow M_{n \times n} (\mathbb{R})$ is orthogonal and $X(t), S(t)$ and $Y(t)$ are analytic.
\end{defn}

Diagonal entries $s_i(E(t))$ of $S(t)$ are called the analytic singular values. Due to the requirement of smoothness, singular values may be negative and also their ordering may be arbitrary.

Let $\mathcal{A}_{m, n}([a,b])$ denote the set of matrix functions $A(t)$ such that for each $t$, $ A(t) \in M_{m \times n}(\mathbb{R})$, and all the entries of $A(t)$ are  real analytic function on $[a,b]$.

\begin{thm}[ Theorem 1, \cite{sing}]\label{asvd}
If $E(t) \in \mathcal{A}_{m,n}([a,b])$, then there exists an analytic singular value decomposition on $[a,b]$.
\end{thm}

The following theorem is an extended version of the preceding theorem for the set of all $m \times n$ complex matrices.

\begin{thm}[Theorem 3.1, \cite{gin}]\label{asvd1}
Let $E(t):[a,b]\rightarrow M_{m\times n}(\mathbb{C})$ be an $m\times n$ matrix function, not identically zero, with analytic entries on $[a,b]$, where $-\infty \leq a<b \leq \infty$. Then, $E(t)$ can be factored as $$E(t)=X(t)S(t)Y(t)^\ast,$$ where $X(t)$ and $Y(t)$ are unitary matrix functions, with all entries are analytic on $[a,b]$, of order  $m \times m$ and $n \times n$, respectively. For each $t$,  $S(t)$ is a diagonal matrix with diagonal entries and  $s_1(E(t)),s_2(E(t)),\dots,s_n(E(t))$ analytic on $[a,b]$.
\end{thm}

Here the columns of $X(t)$ consist of normalized eigenvectors of $E(t)E(t)^\ast$, and the columns of $Y(t)$ consist of normalized eigenvectors of $E(t)^\ast E(t)$. Also, $e_i(t)=s_i(E(t))^2$, where $e_i(t)$ are the eigenvalues of $E(t)^\ast E(t)$, are analytic on $[a,b]$. The functions $s_i(E(t))$ are called the analytic singular values of $E(t)$.

\section{Dense subsets of matrices with distinct  eigenvalues and distinct singular values }\label{cls-eigen-sing}

In this section our main objectives are the following:  Given  $\Omega$, a subset of $M_{n\times n}( \mathbb{C})$, we prove some of the results about the existence of  dense subsets $\Omega_s$ of $\Omega$ such that all the matrices in $\Omega_s$ have distinct eigenvalues. For $m \geq n $, we consider the counterpart of these problem for the singular values.  To facilitate understanding, we divide this section into two subsections, one for the results about the eigenvalues and others for the results about the singular values.

\subsection{ Eigenvalue case}\label{sec-eig-clas}

	For a function $F:\mathcal{D}\rightarrow M_{n\times n}(\mathbb{C})$, let us define $\mathcal{Z}(\mathcal{D})=\{x\in \mathcal{D} : F(x) \mbox{ has repeated eigenvalues} \}$, where $\mathcal{D}$ is a subset of $\mathbb{C}$. In the following theorem, we prove that if $\mathcal{D}$ is an open, connected subset of $\mathbb{C}$ and the entries of $F$ are analytic functions on $\mathcal{D}$, then either $\mathcal{Z}(\mathcal{D})=\mathcal{D}$ or $\mathcal{Z}(\mathcal{D})$ has no limit points.

\begin{thm}\label{a4}
    Let $\mathcal{D}$ be an open connected subset of $\mathbb{C}$, and $F:\mathcal{D} \rightarrow M_{n \times n} (\mathbb{C})$ be a  function whose entries are analytic functions on $\mathcal{D}$.   Then, either $\mathcal{Z}(\mathcal{D})=\mathcal{D}$ or $\mathcal{Z}(\mathcal{D})$ has no limit points.

\end{thm}

\begin{proof}
    Let $$p_x(y)=\det \left( yI-F(x) \right)=y^n+ \sum_{k=1}^n h_k(x)y^{n-k}.$$ Then, $p_x(y)$ is a polynomial in $y$,  and $h_k(x)$ is analytic, for each $1  \leq k  \leq n$. For each fixed $x\in \mathcal{D}$, the eigenvalues of the matrix $F(x)$ are the roots of the polynomial $p_x(y)=0$. So, if $p_x(y)=0$ has multiple roots, then  $F(x)$ has repeated eigenvalues.  Now, by Theorem \ref{res}, $p_x(y)=0$ has multiple roots if and only if $\Res(p_x(y),p'_x(y))=0$, where $p'_x(y)=ny^{n-1} + \sum\limits_{k=1}^{n-1}(n-k)h_k(x)y^{n-k-1}.$ Also, it is easy to see that, $\Res(p_x(y),p'_x(y))$ is an analytic function of $x$ on $\mathcal{D}$. Hence, the zero set of  $\Res(p_x(y),p'_x(y))$ is either $ \mathcal{D}$ or an isolated subset of $\mathcal{D}$.

    If $\Res(p_x(y),p'_x(y))=0$ for all $x\in \mathcal{D}$, then the polynomial  $p_x(y)=0$ has multiple roots for all $x \in \mathcal{D}$. Thus, all matrices in $F(\mathcal{D})$ has repeated eigenvalues, so $\mathcal{Z}(\mathcal{D})=\mathcal{D}$. If the zero set of $\Res(p_x(y),p'_x(y))=0$ is isolated in  $ \mathcal{D}$, then the set $\mathcal{Z}(\mathcal{D})$ has no limit point.
\end{proof}

In the following theorem, we show that if  $\mathcal{D}$ is an open interval in $\mathbb{R}$ and the entries of the function $F$ are analytic functions on $\mathcal{D}$, then the conclusion of the above result holds  true for $\mathcal{Z}(\mathcal{D})$.

\begin{thm}
    Let $\mathcal{D}$ be an open interval in $\mathbb{R}$, and $F:\mathcal{D} \rightarrow M_{n \times n} (\mathbb{C})$ be a  function whose entries are analytic functions on $\mathcal{D}$.   Then, either $\mathcal{Z}(\mathcal{D})=\mathcal{D}$ or $\mathcal{Z}(\mathcal{D})$ has no limit points.
\end{thm}

\begin{proof}
    The proof is similar to that of Theorem \ref{a4} (Using  Theorem \ref{pow-uniq-rudin}).
\end{proof}

The next theorem is analogous of Theorem \ref{a4}. Here we consider the entries of the function $F$ are polynomials, and $\mathcal{D}$ is an arbitrary subset of $\mathbb{C}$. The idea of the proof is similar to that of in \cite[Theorem 1]{hart}. This theorem is vital to prove some of the theorems of this section.

\begin{thm}\label{pd}
    Let $F:\mathcal{D} \subseteq \mathbb{C} \rightarrow M_{n \times n} (\mathbb{C})$ be a function defined by $F(x)=\sum\limits_{k=0}^p A_k x^k,$ where $A_k \in M_{n \times n}(\mathbb{C})$ for $k=0,1,\dots,p$. Then, either $\mathcal{Z}(\mathcal{D})=\mathcal{D}$ or $\mathcal{Z}(\mathcal{D})$ is finite.
\end{thm}

\begin{proof}
    It is easy to see that, $(i,j)$th entry of the matrix $F(x)$   is $\sum\limits_{k=0}^p a_{ij}^{(k)} x^k$, where $a_{ij}^{(k)}$ is the $(i,j)$th entry of the matrix $A_k$ for $i,j=1,2,\dots,n$.

    Set $$p_x(y)=\det \left( yI-F(x) \right)=y^n+ \sum_{k=1}^n q_k(x)y^{n-k}.$$ Then $p_x(y)$ is a polynomial in $y$, and each coefficient $q_k(x)$ is a polynomial in $x$. The eigenvalues of $F(x)$ are the roots of the equation  $p_x(y)=0$ for each $x\in \mathcal{D}.$

    Now, the matrix $F(x)$ has repeated eigenvalues if and only if the polynomial $p_x(y)$ has repeated roots. By Theorem \ref{res}, $p_x(y)=0$ has repeated roots if $\Res(p_x(y),p'_x(y))=0$, where $p'_x(y)=ny^{n-1} + \sum\limits_{k=1}^{n-1}(n-k)q_k(x)y^{n-k-1}$. As $\Res(p_x(y),p'_x(y))$ is a polynomial in $x$, hence, the zero set of  $\Res(p_x(y),p'_x(y))$ is either $ \mathcal{D}$ or a finite subset of $\mathcal{D}$.

    If $\Res(p_x(y),p'_x(y))=0$ for all $x\in \mathcal{D}$, then the polynomial  $p_x(y)=0$ has multiple roots for all $x \in \mathcal{D}$. Thus, all the matrices in $F(\mathcal{D})$ has repeated eigenvalues, so $\mathcal{Z}(\mathcal{D})=\mathcal{D}$. If the zero set of $\Res(p_x(y),p'_x(y))$ is finite in  $ \mathcal{D}$, then $\mathcal{Z}(\mathcal{D})$ is finite too.\end{proof}

In the following lemma, we show that if the closure of a subset of the set of all $n\times n$ complex matrices contains at least one matrix whose eigenvalues are distinct, then the subset also contains matrices whose eigenvalues are distinct.

\begin{lem}\label{dis}
Let $\Omega$ be a subset of  $M_{n \times n}(\mathbb{C})$, and $\Omega_d$ be the matrices in $\Omega$ having distinct eigenvalues. If $\overline{\Omega}$ contains at least one matrix with distinct eigenvalues, then $\Omega_d$ is non empty.
\end{lem}

\begin{proof}
If $\Omega$ contains a matrix with distinct eigenvalues, then $\Omega_d$ is non-empty.
Let $A$ be a limit point of $\Omega$ whose eigenvalues are distinct. Then there exists a sequence $\{A_m\}$ in $\Omega$ such that $\lim\limits_{m\rightarrow \infty} A_m= A$. Now, if all the matrices $A_m$ have repeated eigenvalues, then, by Theorem \ref{closed},  $A$  has repeated eigenvalues. Hence, $\{A_m\}$ contains matrices having distinct eigenvalues. Hence $\Omega_d$ is non empty.
\end{proof}

Next lemma shows that, if the entries of the function $F$ are polynomials, then either all matrices in $F(\mathcal{D})$ are singular or there are finite number of matrices in $F(\mathcal{D})$ which are singular. To avoid ambiguities, let us assume the set $\mathcal{D}$ is an infinite set.

\begin{lem}\label{zr}
	Let $F: \mathcal{D}\subseteq \mathbb{C} \rightarrow M_{n\times n}(\mathbb{C})$ be a function defined by  $F(x)=\sum\limits_{k=0}^p A_k x^k$, where $A_k \in M_{n \times n}(\mathbb{C})$. Then,  either all the matrices in $F(\mathcal{D})$ are singular or only finitely many of them are singular.
\end{lem}
\begin{proof}
Let us rewrite the $(i, j)$th  entry of the matrix $F(x)$ as follows:  $ \sum\limits_{k=0}^p a_{ij}^{(k)} x^k$, where $a_{ij}^{(k)}$ is the $(i,j)$th entry of $A_k$ for $i,j=1,2,\dots,n$.
As $\det(F(x))$ is a polynomial in $x$,  hence either $\det(F(x))=0$ for all $x \in \mathcal{D}$ or $\det(F(x))=0$ for finitely many $x \in \mathcal{D}$. If $\det(F(x))=0$ for all $x$ in $\mathcal{D}$, then all the matrices in $F(\mathcal{D})$ are singular, and if $\det(F(x))=0$ for finite number of $x$ in $ \mathcal{D}$, then $F(\mathcal{D})$ is finite.
\end{proof}

The next theorem is a generalization of  \cite[Corrollary 1]{hart}.


 \begin{thm}\label{a8}
Let $\Omega$ be a convex subset of  $M_{n \times n}(\mathbb{C})$, and $\Omega_s$ be the set of all nonsingular matrices in $\Omega$ with distinct  eigenvalues. If $\overline{\Omega}$ contains at least one nonsingular matrix with distinct eigenvalues, then $\Omega_s$ is dense in $\Omega$.
\end{thm}

\begin{proof}
Let $A$ be a nonsingular matrix in $\overline{\Omega}$ with distinct eigenvalues. Then, there exists a sequence $\{A_m\}$ in $\Omega$ such that $\{A_m\}$ converges to $A$. If all the matrices $A_m$ are singular, then $A$ is also singular. So the sequence $\{A_m\}$ contains nonsingular  matrices. By Theorem \ref{closed}, the sequence $\{A_m\}$ contains matrices whose eigenvalues are distinct. Let $A_r$ and $A_s$ be two matrices in the sequence $\{A_m\}$ such that $A_r$ is non-singular, and $A_s$ has distinct eigenvalues.

Set $F(t)=(1-t)A_r+tA_s$, where $0 \leq t \leq 1$. Now, $F([0,1])$ contains a non-singular matrix, and a matrix whose eigenvalues are distinct. Hence, by Theorem \ref{pd},  $F(t)$ has repeated eigenvalues only for finitely many $t^{'}s$ in $[0,1]$. By Lemma \ref{zr},  $F(t)$  is singular only for finitely many $t^{'}s$ in $ [0, 1]$. Thus $F([0,1])$ contains nonsingular matrices with distinct eigenvalues. Hence $\Omega_s \neq \phi$.

Let  $A \in \Omega$ and  $B \in \Omega_s$. Set $E(t)=(1-t)A+tB=t(B-A)+A$, for $0 \leq t \leq 1$. By Theorem \ref{pd},  $E(t)$ has repeated eigenvalues only for finitely many $t'$s,  and, by Lemma \ref{zr},  $E(t)$ is singular only for finitely many $t'$s.
Assume that  $E(t)$ has repeated eigenvalues for $t=t_1,t_2,\dots,t_p$, and  $E(t)$ is singular for for $t=t_{p+1},\dots,t_{p+q}$. Let $\mathcal{L}=\{t_i>0: i \in \{1,2,\dots,{p+q}\}\} $. If $\mathcal{L}$ is nonempty, then, define $s=\min \mathcal{L}$, otherwise choose $s$ to be any real number in the interval $(0,1)$.  Then, for any $t \in (0,s)$, the matrix $E(t)$ is nonsingular and has distinct eigenvalues.  Hence, for any $\epsilon >0$, the open ball $\mathrm{B}(A;\epsilon)$ has nonempty intersection with $\Omega_s$. As $A\in \Omega$ is arbitrary,  hence $\Omega_s$ is dense in $\Omega$.
\end{proof}

  The idea of the following theorem is to extend the idea of the previous theorem viz., instead convex combination of matrices, one can look at the arbitrary polynomial combination.

\begin{thm}\label{conv-gen-eig}
Let $\Gamma$ be a subset of $ M_{n \times n}(\mathbb{C}) $ such that, if  $A$ and $B$ are in  $\Gamma$, then there exists a polynomial $p(x)=\sum\limits_{i=0}^k A_i x^i $ defined on $[0,1]$ such that $p(0)=A , p(1)=B$ and $p([0,1]) \subset \Gamma$. Let $\Gamma_s$ be the set of all matrices in $\Gamma$ which are nonsingular and have distinct eigenvalues. Then, $\Gamma_s$ is dense in $\Gamma$ if and only if $\overline{ \Gamma}$ contains a nonsingular matrix whose eigenvalues are distinct.
\end{thm}

\begin{proof}
Here only if condition is easy to verify. Now, if $\overline{\Gamma}$ contains a nonsingular matrix whose eigenvalues are distinct, then, by the proof of  Theorem \ref{a8}, it is clear that $\Gamma_s$ is nonempty. Let  $A$ be an element of $\Gamma$, and  $B$  be an element of $\Gamma_s$. Let $p(x)$ be a polynomial $\sum\limits_{i=0}^k A_i x^i$ in $[0,1]$ such that $p(0)=A$, $p(1)=B$ and $p([0,1])\subset \Gamma$. Rest of the proof to similar to that of Theorem \ref{a8}.
\end{proof}

\subsection{Singular value case}\label{subsec3.2}

In this section, we shall extend some of the results of section \ref{sec-eig-clas} for the singular values of matrices. For a function $F:\mathcal{U}\rightarrow M_{m\times n} (\mathbb{C}),$ let us define $\mathcal{Y}(\mathcal{U}) = \{x\in \mathcal{U} :  F(x)\mbox{ has repeated singular values}\}$, where $\mathcal{U}$ is a subset of $\mathbb{R}$. The next theorem is a counter part of Theorem \ref{a4} for the singular values of matrices.


\begin{thm}\label{a7}
Let $\mathcal{U}$ be an open interval in $\mathbb{R}$, and $F:\mathcal{U} \rightarrow M_{m \times n} (\mathbb{C})$ be a function whose entries are analytic functions on $\mathcal{U}$. Then, either $\mathcal{Y}(\mathcal{U})=\mathcal{U}$ or $\mathcal{Y}(\mathcal{U})$ has no limit points.
\end{thm}

\begin{proof}
The singular values of an $m\times n$ complex matrix $A$ are  positive square roots of the eigenvalues of  $A^{\ast}A$. Set $$p_x(y)=\det \left( yI-F(x)^{\ast}F(x) \right)=y^n+ \sum_{k=1}^n h_k(x)y^{n-k}.$$ Then, $p_x(y)$ is a polynomial in $y$,  and $h_k(x)$ is analytic, for each $1  \leq k  \leq n$. For each fixed $x\in \mathcal{U}$, the singular values of $F(x)$ are the positive square roots of the roots of the polynomial $p_x(y)=0$. So, if $p_x(y)=0$ has multiple roots, then  $F(x)$ has repeated singular values.  Now, by Theorem \ref{res}, $p_x(y)=0$ has multiple roots if and only if $\Res(p_x(y),p'_x(y))=0$, where $p'_x(y)=ny^{n-1} + \sum\limits_{k=1}^{n-1}(n-k)h_k(x)y^{n-k-1}$. Also, it is easy to see that, $\Res(p_x(y),p'_x(y))$ is an analytic function of $x$ on $\mathcal{U}$. Hence, by Theorem \ref{pow-uniq-rudin}, the zero set of  $\Res(p_x(y),p'_x(y))$ is either the set $ \mathcal{U}$  itself or an isolated subset of $\mathcal{U}$.

If $\Res(p_x(y),p'_x(y))=0$ for all $x\in \mathcal{U}$, then the polynomial  $p_x(y)=0$ has multiple roots for all $x \in \mathcal{U}$. Thus, all matrices in $F(\mathcal{U})$ has repeated singular values, so $\mathcal{Y}(\mathcal{U})=\mathcal{U}$. If the zero set of $\Res(p_x(y),p'_x(y))=0$ is isolated in  $ \mathcal{U}$, then $\mathcal{Y}(\mathcal{U})$ has no limit points.
\end{proof}

In the next theorem, we consider the entries of $F(x)$ are polynomials, instead of analytic functions. In this case, the domain of the function $F$ need not be an open interval.
\begin{thm}\label{fin}
Let $F:\mathcal{U} \subseteq \mathbb{R} \rightarrow M_{m \times n} (\mathbb{C})$ be a function defined by $F(x)=\sum\limits_{k=0}^p A_k x^k,$ where $A_k \in M_{m \times n}(\mathbb{C})$ for $k=0,1,\dots,p$. Then, either $\mathcal{Y}(\mathcal{U})=\mathcal{U}$ or $\mathcal{Y}(\mathcal{U})$ is finite.
\end{thm}

\begin{proof}
It is easy to see that, $(i,j)$th entry of the matrix $F(x)$   is $\sum\limits_{k=0}^p a_{ij}^{(k)} x^k$, where $a_{ij}^{(k)}$ is the $(i,j)$th entry of $A_k$ for $i=1,2,\dots,m$ and $j=1,2,\dots,n$.

Set $$p_x(y)=\det \left( yI-F(x)^{\ast}F(x) \right)=y^n+ \sum_{k=1}^n q_k(x)y^{n-k}.$$ Then $p_x(y)$ is a polynomial in $y$, and each coefficient $q_k(x)$ is a polynomial in $x$. The singular values of $F(x)$ are the positive square roots of the roots of the equation  $p_x(y)=0$ for each $x\in \mathcal{U}.$

Now for each $x$ in $\mathcal{U}$, $F(x)$ has repeated singular values if and only if $p_x(y)=0$ has repeated roots. By Theorem \ref{res}, $p_x(y)=0$ has repeated roots if $\Res(p_x(y),p'_x(y))=0$, where $p'_x(y)=ny^{n-1} + \sum\limits_{k=1}^{n-1}(n-k)q_k(x)y^{n-k-1}.$ Rest of proof is similar to that of Theorem \ref{pd}.
\end{proof}

\begin{rmk}\label{mo-ans}
	{\rm Now, let us answer the following question which we discussed in the introduction:  For an $n \times n $ complex matrix  $A$ and $\epsilon>0$, does there exist  a complex symmetric(not Hermitian) matrix $S$ such that $\Vert S\Vert \leq \epsilon$ and $A+S$ has only distinct singular values?  Let $B=\diag(b_1,b_2,\dots,b_n)$ be an $n\times n$ real diagonal matrix with distinct diagonal entries. Consider the function $F(t)=A+tB$  defined on $[0,1]$. Now,by Theorem \ref{fin},  $F([0,1])$ contains only finitely many matrices whose singular values are repeated. So we can choose an $c$ in $(0,1)$ such that $A+cB$ has distinct singular values and $\vert c \vert \Vert B \Vert \leq \epsilon$. Hence, $S=cB$ solves the problem.}
\end{rmk}
The following example shows that,  in Theorem \ref{a7} and  Theorem \ref{fin}, we may not be able to extend the domain $\mathcal{U}$ of the function $F(x)$ from a subset of $\mathbb{R}$ to a subset of $\mathbb{C}$.

\begin{exam}
Let us consider the function $F:\mathbb{C}\rightarrow M_{3\times 2}(\mathbb{C})$ defined by $F(z)=
\begin{bmatrix}
z&0\\
0&z\\
0&0
\end{bmatrix}
.$
Then, each entry of $F(z)$ is a polynomial in $z$, which is also an analytic function of $z$. Now,  $$F(z)^\ast F(z)=
\begin{bmatrix}
\vert z\vert^2 & 0\\
0& \vert z \vert^2
\end{bmatrix}
.$$
So the diagonal entries of $F(z)^\ast F(z)$ are neither polynomials in $z$ nor analytic functions in $z$ on $\mathbb{C}$. So the idea of the proofs Theorem \ref{a7} and Theorem \ref{fin} may not helpful.
\end{exam}

Using Theorem \ref{fin}, for a convex subset of $m\times n$ complex matrices, we establish a necessary and sufficient condition for the existence of a dense subset of matrices with distinct singular values. This result is a counter part of Corollary 1 of \cite{hart} for the singular values.

\begin{thm}\label{conv-dist-sing}
 Let $\Omega$ be a convex subset of  $M_{m \times n}(\mathbb{C})$, and $\Omega_d$ be the matrices in $\Omega$ having distinct singular values. Then $\Omega_d$ is dense in $\Omega$ if and only if $\Omega_d$ is non empty.
\end{thm}

\begin{proof}

If $\Omega_d$ is dense in $\Omega$, then $\Omega_d$ is non empty. Suppose that $\Omega_d$ is non empty. Let $A\in \Omega$ and $B\in \Omega_d$. Let $E(t)=(1-t)A+tB$ where $0 \leq t \leq 1$. Then $E(t) \subseteq \Omega$, as $\Omega$ is a convex subset.
 Now,  $B$ is in $E([0,1])$, and $B$  has distinct singular values. So $\mathcal{Y}([0,1])$ is a proper subset of $[0,1]$. As the entries  of $E(t)$ are  polynomials in $t$, so, by Theorem \ref{fin}, $\mathcal{Y}([0,1])$ is finite.





Let $\mathcal{L}=\{t>0 : t\in \mathcal{Y}([0,1])\}$. If $\mathcal{L}$ is nonempty, then, define $s=\min \mathcal{L}$, otherwise choose $s$ to be any real number in $(0,1)$. Now, each matrix in $E((0,s))$ has distinct singular values. Hence for arbitrary $\epsilon>0$, the open ball $\mathrm{B}(A;\epsilon)$ has nonempty intersection with $\Omega_d$.  Hence $\Omega_d$ is dense in $\Omega$.
\end{proof}

The following corollary gives a necessary and sufficient condition under which a subspace of $M_{m \times n}(\mathbb{C})$ has a  dense subset, which is a simple consequence of the previous theorem.

\begin{cor}\label{sub}
Let $\Omega$ be a subspace of  $M_{m \times n}(\mathbb{C})$, and $\Omega_d$ be the matrices in $\Omega$ having distinct singular values. Then $\Omega_d$ is dense in $\Omega$ if and only if $\Omega_d $ is nonempty.
\end{cor}

The following lemma gives a condition, which can confirm the existence of matrix whose singular values are distinct, in a subset of a $m\times n$ complex matrices.

\begin{lem}\label{boun-dist-sing}
Let $\Omega$ be a subset of  $M_{m \times n}(\mathbb{C})$, and $\Omega_d$ be the set of all matrices in $\Omega$ having distinct singular values. If $\overline{\Omega}$ contains at least one matrix having distinct singular values, then $\Omega_d $ is nonempty.
\end{lem}

\begin{proof}
 If $\Omega$ contains a matrix, having distinct  singular values, then $\Omega_d$ is non empty.
Let $A$ be a limit point of $\Omega$, whose singular values are distinct. Then there exists a sequence $\{A_p\}$ in $\Omega$ which converges to $ A$. Now, if $A_p$ has repeated singular values for all $p \in \mathbb{N}$, then $A_p^{\ast}A_p$ has repeated eigenvalues for all $p \in \mathbb{N}$. Now,
$ \lim\limits_{p\rightarrow\infty} A_p^{\ast}A_p=A^{\ast}A.$ Again, each $A_p^{\ast}A_p$ has repeated eigenvalues, so, by Theorem \ref{closed}, $A^{\ast}A$ has repeated eigenvalues. Hence $A$ has repeated singular values, which contradict the assumption that $A$ has distinct singular values. Hence, $\{A_p\}$ must contains matrices, whose singular values are distinct. So $\Omega_d$ is nonempty.
\end{proof}

The following corollary is an analogous of Theorem \ref{conv-dist-sing}, where we weaken the condition $\Omega_d$ is nonempty.

\begin{cor}\label{13}
Let $\Omega$ be a convex subset of  $M_{m \times n}(\mathbb{C})$ and $\Omega_d$ be the matrices in $\Omega$ having distinct singular values. Then $\Omega_d$ is dense in $\Omega$ if and only if $\overline{\Omega}$ contains a matrix having distinct singular values.
\end{cor}

\begin{proof}
Proof follows from Theorem \ref{conv-dist-sing} and Lemma \ref{boun-dist-sing}.
\end{proof}

The next theorem is a consequence of Theorem \ref{fin} and Lemma \ref{boun-dist-sing}.
\begin{thm}
    Let $\Gamma$ be a subset of $ M_{m \times n}(\mathbb{C}) $, and  let $\Gamma_d$ be the set of all matrices in $\Gamma$ whose singular values are distinct.  Suppose that for given two matrices $A$ and $B$  in $\Gamma$,  there exists a polynomial $p(x)=\sum\limits_{i=0}^k A_i x^i $ on $[0,1]$ such that $p(0)=A , p(1)=B$ and $p([0,1]) \subset \Gamma$. Then, $\Gamma_d$ is dense in $\Gamma$ if and only if $\overline{\Gamma}$ contains a matrix whose singular values are distinct.
\end{thm}

\begin{proof}

    The proof similar to that of Theorem \ref{conv-dist-sing}.
\end{proof}


For a matrix $A\in M_{m \times n} (\mathbb{C})$, $A=(a_{ij})$ has a nonzero diagonal, if there exists an injective function $f:\{1,2,\dots,n\}\rightarrow \{1,2,\dots,m\}$ such that $a_{f(i)i}$ is nonzero for $i=1,2,\dots,n $ \cite{minc-per}. Let $\mathcal{P}$ denote the set of all $m\times n$ matrices whose entries are either $0$ or $1$. For an $m\times n$ complex matrix $A=(a_{ij})$, define $\tilde{A}= (\tilde{a}_{ij})$, as follows:

\[
\tilde{a}_{ij}=\left\{
\begin{array}{cc}
1, &\hspace{5mm} a_{ij}\neq 0,\\
0, &\hspace{5mm} a_{ij}=0.
\end{array}
\right.
\]

For a matrix $P=(p_{ij})\in \mathcal{P}$,  define $\mathcal{S}(P)=\{A\in M_{m \times n} (\mathbb{C}): \tilde{a}_{ij} \leq p_{ij} \}.$ It is easy to verify that $\mathcal{S}(P)$ is a subspace of $M_{m \times n} (\mathbb{C})$.

It is clear from the previous theorems that $\mathcal{S}(P)$ contains a dense subset of matrices, whose singular values are distinct if and only if $\overline{\mathcal{S}(P)}$ includes a matrix having distinct singular values.  For a matrix, $P\in \mathcal{P}$, let $\mathcal{S}(P)_{d}$ denote the set of all matrices in $\mathcal{S}(P)$ having distinct singular values. In the next theorem, we give a necessary and sufficient condition for the subset $\mathcal{S}(P)_{d}$ to be dense in  $\mathcal{S}(P)$. 

\begin{thm}\label{zero-patt-sing}
Let $P\in \mathcal{P}$. Then $\mathcal{S}(P)_{d}$ is dense in $\mathcal{S}(P)$ if and only if either $P$ or $P_{ii}$ has a nonzero diagonal for some $i=1,2,\dots,n$, where $P_{ii}$ is obtained by deleting the $i$th row and the $i$th column of $P$.
\end{thm}

\begin{proof}
Let $\mathcal{S}(P)_d$ be dense in $\mathcal{S}(P)$. Then $\mathcal{S}(P)$ contains a matrix $A$ whose singular values are distinct.
If all the $n$ singular values of $A$ are nonzero, then the rank of the matrix $A$ is $n$. Thus $A$ has an $n\times n$ sub matrix $A_s$ such that $\det(A_s)\neq 0$. Let the  $i_k$th row of $A$ be the $k$th row of $A_s$. Now, by definition of determinant, it is clear that $A_s$ must has a nonzero diagonal.  Let  $a_{\sigma(k)k}$, where $k=1,2,\dots,n$, be the elements of a nonzero diagonal,  where   $\sigma$ is a permutation on $\{1,2,\dots,n\}$. The entries  $a_{\sigma(k)k}$ in $A_s$ and  $a_{i_{\sigma(k)}k}$ in $A$ are the same. The function, which maps $k$ to $i_k$ is an injective function from $\{1,2,\dots,n\}$ to $\{1,2,\dots,m\}$. Hence the function $f:\{1,2,\dots,n\} \rightarrow \{1,2,\dots,m\}$ defined by  $f(k)=i_{\sigma(k)}$ is an injective function. Thus $A$ has a nonzero diagonal whose elements are $a_{f(k)k}$ where $1,2,\dots,n$, and hence the matrix $P$ has a nonzero diagonal too.

If the matrix $A$ has exactly $n-1$ distinct nonzero singular values, then the rank of $A$ is $n-1$. So $A$ has an $(n-1)\times (n-1)$ sub matrix $A_s$ such that $\det(A_s) \neq 0$. Now applying the same argument for the matrix $A_s$ as above, we get a nonzero diagonal in the matrix $A_s$. Hence the matrix  $P_{ii}$ has a nonzero diagonal, for some $i \in \{1, \dots , n\}$.

Conversely,  suppose that $P=(p_{ij})$ has a nonzero diagonal with the diagonal entries $a_{f(i)i}$ for $i=1,2,\dots,n$, where $f:\{1,2,\dots,n\} \rightarrow \{1,2,\dots,m \}$ is an injective function. Let us construct the matrix $A=(a_{ij})$ as follows:

\[
a_{ij}=\left\{
\begin{array}{cc}
k, &\hspace{5mm}  \mbox{if } i=f(k) \mbox{~and ~}j=k, \\
0, &\hspace{5mm} \mbox{otherwise}.
\end{array}
\right.
\]
Then $A^\ast A=diag(1,4,\dots,n^2)$.

 If $P$ does not have any nonzero diagonal, then, for some $k$,  the matrix $P_{kk}$ has a nonzero diagonal whose entries are $p_{g(i)i}$ for $i=1,2,\dots,n, i\neq k$, where $g:\{1,2,\dots,n \}\setminus \{k\} \rightarrow \{1,2,\dots,m\}\setminus \{k\}$.  Now, construct the matrix $B=(b_{ij})$ as follows:

\[
b_{ij}=\left\{
\begin{array}{cc}
l, &\hspace{5mm} \mbox{if } i=g(l),~ j=l \mbox{~and~} l\neq k,\\
0, &\hspace{5mm} \mbox{otherwise}.
\end{array}
\right.
\]
Then,  $B^\ast B=diag(1,4,\dots,(k-1)^2,0,(k+1)^2,\dots,n^2)$.

Thus,  in each cases,  there exist matrices in $ \mathcal{S}(P)$ which  has distinct singular values. Hence, by Corollary \ref{sub},  $\mathcal{S}(P)_d$ is dense in $\mathcal{S}(P)$.
 \end{proof}

This theorem is a counter part of  \cite[Theorem 2 ]{hart} for the singular values of rectangular matrices.

\section{Dense subsets with distinct analytic eigenvalues and analytic singular values}
 In this section, we shall establish some of the results related to the denseness of subsets of matrices having distinct analytic eigenvalues/analytic singular values. The results are parallel to that of Section \ref{cls-eigen-sing}.   The following lemma will be useful in the proof of some of the results of this section.

\begin{lem}\label{a1}
Let $f_1(x),f_2(x),\dots, f_n(x)$ be analytic functions of  $x$ on $[a,b]$ and $\mathcal{Y}=\bigcup\limits_{i \neq j} \{x \in [a,b] : f_i(x)=f_j(x)\}$. Then, either $\mathcal{Y}=[a,b]$ or $\mathcal{Y}$ is finite.
\end{lem}

\begin{proof}
	As each $f_i(x)$ is analytic on $[a,b]$, there exists an open interval $\mathcal{U}$,  containing $[a,b]$ such that each $f_i(x)$ is analytic on $\mathcal{U}$.

For $i\neq j$, define $\mathcal{Y}_{ij}=\{x\in \mathcal{U} :f_i(x)=f_j(x)\}$. Now, if the sets $\mathcal{Y}_{ij} \cap [a,b]$ are finite, then $\mathcal{Y}=\cup_{i\neq j} \mathcal{Y}_{ij} \cap [a,b]$ is finite. If  $\mathcal{Y}_{ij} \cap [a,b]$, for some $i$ and $j$, is infinite, then $\mathcal{Y}_{ij} \cap [a,b]$ must have a limit point in $\mathcal{Y}_{ij} \cap [a,b] \subset \mathcal{U}$. Thus, by Theorem \ref{pow-uniq-rudin},  $\mathcal{Y}_{ij}=\mathcal{U}$, and hence $\mathcal{Y}_{ij} \cap [a,b]=[a,b]$. That is, $\mathcal{Y}=\left(\cup_{i \neq j} \mathcal{Y}_{ij}\right) \cap [a,b]=[a,b]$.
\end{proof}

We divide this section into two subsection to facilitate understanding.
\subsection{Analytic eigenvalue case}
In this subsection, at first, we shall prove  a theorem similar to that of Theorem \ref{a4}, using analytic spectral decomposition in $M^{\mathbb{R}}_{n\times n}(\mathbb{C})$.

\begin{thm}\label{a2}
Let $F:[a,b] \rightarrow M^{\mathbb{R}}_{n \times n} (\mathbb{C})$ be a  function whose entries are analytic functions on $[a,b]$. Let $\mathcal{W}$ be the collection of $x$ in $[a,b]$, for which $F(x)$ has repeated eigenvalues. Then, either $\mathcal{W}=[a,b]$ or $\mathcal{W}$ is finite.
\end{thm}

\begin{proof}
As the entries of the function $F(x)$ are analytic on $[a,b]$. By Theorem \ref{a9}, there exists a unitary matrix $U(t)$ analytic on $[a,b]$ such that $$Q(x)=U^{-1}(x)F(x)U(x),$$ where $Q(x)$ is an upper-triangular matrix whose entries are analytic functions of $x$ on $[a,b]$. Since the eigenvalues of $F(x)$ are the diagonal entries of $Q(x)$, so the eigenvalues of $F(x)$ are analytic functions of $x$ on $[a,b]$. Let $e_1(x),e_2(x),...,e_n(x)$ be the eigenvalues of $F(x)$. So $\mathcal{W}=\bigcup\limits_{i\neq j} \{x\in [a,b] : e_i(x)=e_j(x)\}$. Hence, by Lemma \ref{a1}, either $\mathcal{W}$ is finite or $\mathcal{W}=[a,b]$.
\end{proof}

Now, using the previous theorem, we shall establish that if a convex subset $\Omega$ of  $M^{\mathbb{R}}_{n \times n}(\mathbb{C})$ contains a matrix whose eigenvalues are distinct, then the set of all matrices in $\Omega$ with distinct eigenvalues forms a dense subset of $\Omega$.
\begin{thm}\label{a3}
 Let $\Omega$ be a convex subset of  $M^{\mathbb{R}}_{n \times n}(\mathbb{C})$, and $\Omega_d$ be the matrices in $\Omega$ having distinct  eigenvalues. Then $\Omega_d$ is dense in $\Omega$ if and only if $\Omega_d$ is non empty.
\end{thm}

\begin{proof}

If $\Omega_d$ is dense in $\Omega$, then $\Omega_d$ is non empty. Now, suppose that $\Omega_d$ is non empty. Let $A\in \Omega$ and $B\in \Omega_d$. Let $E(t)=(1-t)A+tB$ where $0 \leq t \leq 1$. As $\Omega$ is a convex subset, so $E(t) \subseteq \Omega$.
 Let $\mathcal{W}$ be the subset of $[0,1]$ which contains all the elements $t$ for which $E(t)$ has repeated eigenvalues. Since $E([0,1])$ contains the matrix $B$, and $B$ has distinct eigenvalues,  so $\mathcal{W}$ is a proper subset of $[0,1]$, so, by Theorem \ref{a2}, $\mathcal{W}$ is finite.
Let $\mathcal{L}= \{t>0 : t\in \mathcal{W}\}$. If $\mathcal{L}$ is nonempty, then, define $s=\min \mathcal{L}$, otherwise choose $s$ to be any real number in $(0,1)$. Then, each matrix in $E((0,s))$ has distinct eigenvalues. Hence for arbitrary $\epsilon>0$, the open ball $\mathrm{B}(A;\epsilon)$ has nonempty intersection with $\Omega_d$. Hence $\Omega_d$ is dense in $\Omega$.
\end{proof}

The idea of the proof of the above theorem can be generalized as follows:

\begin{thm}\label{poly-eig-den}
Let $\Gamma$ be a subset of $ M^{\mathbb{R}}_{n \times n}(\mathbb{C}) $ with the properties that, if  $A,B \in \Gamma$, then there exists a function $p(x):[0,1]\rightarrow M^{\mathbb{R}}_{n\times n}(\mathbb{C})$ whose entries are  analytic functions of $x$ on $[0,1]$ such that $p(0)=A , p(1)=B$ and $p([0,1]) \subset \Gamma$. Let $\Gamma_d$ be the matrices in $\Gamma$ whose eigenvalues are distinct. If $\Gamma$ contains at least one matrix with distinct eigenvalues, then $\Gamma_d$ is dense in $\Gamma$.
\end{thm}

\begin{proof} Proof is similar to that of Theorem \ref{a3}.
\end{proof}

\subsection{Analytic singular value case}
In this section, we prove results related to the denseness of the set of all matrices with distinct analytic singular values. The following is useful in the proof of the main result of this section.

\begin{thm}\label{ds}
    Let $F:[a,b] \rightarrow M_{m \times n} (\mathbb{C})$ be a  function whose entries are analytic functions on $[a,b]$. Let $\mathcal{W}$ be the collection of $x$ in $[a,b]$, for which $F(x)$ has repeated singular values. Then, either $\mathcal{W}=[a,b]$ or $\mathcal{W}$ is finite.
\end{thm}

\begin{proof}
    As each entry of $F(x)$ is an analytic function on $[a,b]$,  the function  $F:[a,b]\rightarrow M_{m\times n} (\mathbb{C})$ is an analytic matrix valued function. 
    Thus, by Theorem \ref{asvd1}, there exists an analytic singular value decomposition for $F(x)$ on $[a,b]$. Let $F(x)=U(x)S(x)V(x)^\ast$, where $U(x)$ and $V(x)$ are unitary and $S(x)$ is diagonal. Let $s_1(F(x)), s_2(F(x)),\dots, s_n(F(x))$ be the diagonal entries of $S(x)$. Then the functions $s_1(F(x)), s_2(F(x)),\dots, s_n(F(x))$ are analytic on $[a,b]$. As $\mathcal{W}=\bigcup\limits_{i\neq j} \{x\in [a,b] :s_i(F(x))=s_j(F(x))\}$, hence, by Lemma \ref{a1}, either $\mathcal{W}$ is finite or $\mathcal{W}=[a,b]$.
\end{proof}

As a consequence of the previous theorem, we shall establish that if a convex subset $\Omega$ of  $M_{m \times n}(\mathbb{C})$ contains a matrix whose eigenvalues are distinct, then the set of all matrices in $\Omega$ with distinct eigenvalues forms a dense subset of $\Omega$.

\begin{thm}
    Let $\Omega$ be a convex subset of  $M_{m \times n}(\mathbb{C})$, and $\Omega_d$ be the matrices in $\Omega$ having distinct singular values. Then $\Omega_d$ is dense in $\Omega$ if and only if $\Omega_d$ is non empty.
\end{thm}

\begin{proof}
    The proof is similar to that of Theorem \ref{a3} (Using Theorem \ref{ds}).
\end{proof}

\begin{thm}
Let $\Gamma$ be a subset of $ M_{m \times n}(\mathbb{C}) $ with the properties that, if  $A,B \in \Gamma$, then there exists a function $p(x):[0,1]\rightarrow M_{m \times n}(\mathbb{C})$ whose entries are analytic functions of $x$ on $[0,1]$ such that $p(0)=A , p(1)=B$ and $p([0,1]) \subset \Gamma$. Let $\Gamma_d$ be the matrices in $\Gamma$ whose singular values are distinct. If $\Gamma$ contains at least one matrix with distinct singular values, then $\Gamma_d$ is dense in $\Gamma$.
\end{thm}

\begin{proof} The proof is similar to that of Theorem \ref{poly-eig-den} (Using  Theorem \ref{ds}).
\end{proof}

Using Theorem \ref{asvd}, all the results of this subsection can be proved for $M_{m\times n}(\mathbb{R})$.

\section{Dense subsets of matrices having distinct roots with respect to polynomials}\label{sec-5}

In this section, we define a class of polynomials in terms of the entries of the entries of the matrices, and prove some results  related to the denseness of subsets of matrices for which the polynomials have distinct roots. Let $\mathbb{C}^n_{\sym}$ be the set of all unordered $n$-tuple of complex numbers, and $\mathbb{C}_n[x]$ denote the set of all polynomials of degree $n$. Define the function $r_n: \mathbb{C}_n[x] \rightarrow \mathbb{C}^n  _{\sym}$ such that the image of a polynomial $f$,  $r_n(f)$, is the unordered $n$-tuple whose entries are the roots of the polynomial $f$. Let $\mathcal{P}_k$ denote the set of all functions $p_x:M_{m\times n}(\mathbb{C})\rightarrow \mathbb{C}_k[x]$ defined by $p_x(A)=x^k+\sum\limits_{i=1}^k q_i(A)x^{i-1}$, where each $q_i(A)$ is a polynomial function of the entries of $A$ such that $r_k\left(p_x\left( M_{m \times n}(\mathbb{C})\right)\right)=\mathbb{C}_{\sym}^k$ .

\begin{defn}
For an $m\times n$ complex matrix $A$, and a polynomial $p_x \in \mathcal{P}_k$, we call $z\in \mathbb{C}$ a zero of $A$ with respect to $p_x$ if $p_z(A)=0$.
\end{defn}

The following example shows that each $\mathcal{P}_k$ is nonempty for $1 \leq k \leq mn.$

\begin{exam}
For an $m \times n$ matrix $A$ with $(i,j)$th entry $a_{ij}$, let us consider $A$ as an element of $\mathbb{C}^{mn}$ by the representation $A=(a_{11},a_{12},\dots,a_{1n},a_{21},a_{22},\dots,a_{2n},\dots, a_{m1},a_{m2}, \dots, a_{mn})$. Now, for a fixed $k$ in $\{1,2,\dots,mn\}$, let us define a function $p_x:M_{m\times n}(\mathbb{C})\rightarrow \mathbb{C}_k[x]$ by $p_x(A)=x^k+\sum\limits_{i=1}^k q_i(A)x^{i-1}$, where $(q_1(A),q_2(A),\dots, q_k(A))$ is the first $k$ coordinates of $A$ in $\mathbb{C}^{mn}$. 
So  $p_x\in \mathcal{P}_k$.
\end{exam}

Now, for a  function $F:\mathcal{D} \subseteq \mathbb{C} \rightarrow M_{m\times n}(\mathbb{C})$ and a fixed $p_x \in \mathcal{P}_k$, where $2\leq k \leq mn$, define $\mathcal{Z}_{p_x}(\mathcal{D}) = \{z \in \mathcal{D} :  F(z) \mbox{ has repeated zeros with respect to } p_x \}$.

Next, we shall prove some theorems for the functions in $\mathcal{P}_k$, where $2\leq k\leq mn$, which are similar to some theorems in Section \ref{cls-eigen-sing} and in Remark \ref{b4}, we shall show, how we can use these theorems to prove theorems in Section \ref{cls-eigen-sing}.

\begin{thm}\label{b3}
	Let $\mathcal{D}$ be an open connected subset of $\mathbb{C}$, and $F:\mathcal{D} \rightarrow M_{m \times n} (\mathbb{C})$ be a  function whose entries are analytic functions on $\mathcal{D}$.   Then, for a $p_x\in \mathcal{P}_k$ either $\mathcal{Z}_{p_x}(\mathcal{D})=\mathcal{D}$ or $\mathcal{Z}_{p_x}(\mathcal{D})$ has no limit points.
\end{thm}

\begin{proof}
	Let $$p_x(F(z))=x^k+ \sum_{i=1}^k f_i(F(z))x^{i-1}.$$ Then, by definition of $p_x$, each $f_i(F(z))$ is a polynomial function of the entries of $F(z)$, hence for each $i$, $f_i(F(z))$ is an analytic function of $z$. 
Rest of the proof is similar to that of the Theorem \ref{a4}.
\end{proof}

\begin{thm}\label{a5}
	Let $F:\mathcal{D} \subseteq \mathbb{C} \rightarrow M_{m \times n} (\mathbb{C})$ be a function defined by $F(z)=\sum\limits_{i=0}^s A_i z^i,$ where $A_i \in M_{m \times n}(\mathbb{C})$ for $i=0,1,\dots,s$. Then, for a $p_x$ in $\mathcal{P}_k$, either $\mathcal{Z}_{p_x}(\mathcal{D})=\mathcal{D}$ or $\mathcal{Z}_{p_x}(\mathcal{D})$ is finite.
\end{thm}

\begin{proof}
	Let $$p_x(F(z))=x^k+ \sum_{i=1}^k f_i(F(z))x^{i-1},$$ then, by definition of $p_x$, each $f_i(F(z))$ is a polynomial in $z$. 
Rest of the proof is similar to that of Theorem \ref{pd}.
\end{proof}



\begin{thm}\label{a6}
	Let $\Omega$ be a convex subset of  $M_{m \times n}(\mathbb{C})$, and $\Omega_d$ be the matrices in $\Omega$ having distinct zeros with respect to a fixed $p_x$ in $\mathcal{P}_k$. Then $\Omega_d$ is dense in $\Omega$ if and only if $\Omega_d$ is non empty.
\end{thm}


The following theorem can be proved using Theorem \ref{a5} and Theorem \ref{a6}.

\begin{thm}\label{b2}
	Let $\Gamma$ be a subset of $ M_{m \times n}(\mathbb{C}) $ with the properties that, if  $A,B \in \Gamma$, then there exists a polynomial $q(z)=\sum\limits_{i=0}^s A_i z^i$ on $[0,1]$ such that $q(0)=A , q(1)=B$ and $q([0,1]) \subset \Gamma$, where each $A_i \in M_{m \times n}(\mathbb{C})$. Let $\Gamma_d$ be the matrices in $\Gamma$ whose zeros are distinct with respect to a fixed $p_x$ in $ \mathcal{P}_k$. Then, $\Gamma_d$ is dense in $\Gamma$ if and only if $\Gamma_d$ is nonempty.
\end{thm}

\begin{rmk}\label{b4}
	{\rm It is easy to see that  the characteristic polynomial of an $n \times n$ matrix $A$,  and for an $m \times n$ matrix $A$ the polynomial $\det(xI-A^\ast A)$ are in $\mathcal{P}_n$.  Hence the results of this section generalizes the results of Section 3.}


\end{rmk}

\textbf{Acknowledgment:}  M. Rajesh Kannan would like to thank the Department of Science and Technology, India, for financial support through the projects MATRICS (MTR/2018/000986) and Early Career Research Award (ECR/2017/000643) .

\bibliographystyle{plain}
\bibliography{him-raj}

\end{document}